\documentclass[12pt]{article}

\usepackage{amssymb}
\usepackage{amsxtra}
\usepackage{amsmath}
\usepackage{amsthm}
\usepackage{stmaryrd}
\usepackage[all]{xy}

\DeclareMathOperator{\PGL}{PGL}
\DeclareMathOperator{\SL}{SL}

\DeclareMathOperator{\Spin}{Spin}
\DeclareMathOperator{\HSpin}{HSpin}
\DeclareMathOperator{\PGO}{PGO}
\DeclareMathOperator{\MM}{M}

\DeclareMathOperator{\Sl}{\mathfrak{sl}}
\DeclareMathOperator{\spin}{\mathfrak{spin}}
\DeclareMathOperator{\HH}{H}
\DeclareMathOperator{\Cliff}{C}

\DeclareMathOperator{\Cent}{Cent}

\DeclareMathOperator{\End}{End\,}

\DeclareMathOperator{\Der}{Der\,}

\DeclareMathOperator{\Br}{Br}

\DeclareMathOperator{\Trd}{Trd}

\DeclareMathOperator{\id}{id}

\DeclareMathOperator{\hyp}{hyp}

\DeclareMathOperator{\ZZ}{{\mathbb Z}}

\theoremstyle{plain}

\newtheorem*{thm*}{Theorem}
\newtheorem{thm}{Theorem}

\newtheorem*{cor*}{Corollary}

\begin{document}

\title{A rational construction of Lie algebras\\of type $E_7$}

\author{Victor Petrov\thanks{This research is supported by 
Russian Science Foundation grant 14-21-00035}}

\maketitle

\begin{abstract}
We give an explicit construction of Lie algebras of type $E_7$ out of a Lie algebra of type $D_6$ with some restrictions. Up to odd degree extensions, every Lie algebra of type $E_7$ arises this way. For Lie
algebras that admit a $56$-dimensional representation we provide a more symmetric construction
based on an observation of Manivel; the input is seven quaternion algebras subject to some relations.
\end{abstract}

%


\section{Introduction}

In \cite{Tits90} Jacques Tits wrote the following: ``It might be worthwile
trying to develop a similar theory for strongly inner groups of type $E_7$.
For instance, can one give a general construction of such groups showing that
there exist anisotropic strongly inner $K$-groups of type $E_7$ as soon as
there exist central division associative $16$-dimensional $K$-algebras of order
$4$ in $\Br K$ whose reduced norm is not surjective?''

The goal of the present paper is to give such (and much more general)
construction. We deal with Lie algebras; of course, the corresponding group is
just the automorphism group of its Lie algebra. By \emph{rational} constructions
we mean those not appealing to the Galois descent, that is involving only terms
defined over the base field.

Let us recall several milestones in the theory. Freudenthal in \cite{Fr53}
gave an elegant explicit construction of the \emph{split} Lie algebra of type $E_7$.
On the language of maximal Lie subalgebras it is a particular case of $A_7$-construction.
Another approach was proposed by Brown in \cite{Br68} (see also \cite{Ga01a}
for a recent exposition); this is an $E_6$-construction. It gives only isotropic Lie
algebras. In full generality $A_7$-construction was described by Allison and
Faulkner in \cite{AF84} as a particular case of a Cayley-Dickson doubling;
generically it produces anisotropic Lie algebras of type $E_7$. Another construction
with this property was discovered by Tits in \cite{Tits90}; in our terms it is an
$A_3+A_3+A_1$-construction. On the other hand, some Lie algebras of type
$E_7$ can be obtained via the Freudenthal magic square, see \cite{Tits66a}
(or \cite{Ga01b} for a particular case).

Our strategy is to define a Lie triple system structure on the ($64$-di\-men\-si\-onal
over $F$) simple module of the even Clifford algebra of a central simple algebra of
degree $12$ with an orthogonal involution under some restrictions. Then the
embedding Lie algebra is of type $E_7$. Our construction is of type $D_6+A_1$.

For Lie algebras that admit a $56$-dimensional representation we provide another
construction based on an observation of Manivel \cite{Ma06} and a result
of Qu\'eguiner--Mathieu and Tignol \cite{QT14}; it is of type $7A_1$.

The author is grateful to Ivan Panin, Anastasia Stavrova and Alexander Luzgarev
for discussing earlier attempts to this work,  to Nikolai Vavilov for pointing out Manivel's paper
\cite{Ma06} and to Anne Qu\'eguiner--Mathieu for her hospitality and fruitful discussions.

\section{Lie triple systems and quaternionic gifts}\label{secLie}

Let $F$ be a field of characteristic not $2$. Recall that a \emph{Lie triple system}
is a vector space $W$ over $F$ together with a trilinear map
\begin{align*}
W\times W\times W&\to W\\
(u,v,w)&\mapsto[u,v,w]=D(u,v)w
\end{align*}
satisfying the following axioms:
\begin{align*}
&D(u,u)=0\\
&D(u,v)w+D(v,w)u+D(w,u)v=0\\
&D(u,v)[x,y,z]=[D(u,v)x,y,z]+[x,D(u,v)y,z]+[x,y,D(u,v)z].
\end{align*}

A \emph{derivation} is a linear map $D\colon W\to W$ such that
$$
D[x,y,z]=[Dx,y,z]+[x,Dy,z]+[x,y,Dz].
$$
The vector space of all derivations form a Lie algebra $\Der(W)$ under
the usual commutator map.

The vector space $\Der(W)\oplus W$ under the map
$$
[D+u,E+v]=[D,E]+D(u,v)+Dv-Eu
$$
form a $\ZZ/2$-graded Lie algebra called the \emph{embedding} Lie
algebra of $W$. Conversely, degree $1$ component of any $\ZZ/2$-graded
Lie algebra is a Lie triple system under the triple commutator map.

Consider a $\ZZ$-graded Lie algebra
$$
L=L_{-2}\oplus L_{-1}\oplus L_0\oplus L_1\oplus L_2
$$
with one-dimensional components $L_{-2}=Ff$, $L_2=Fe$, such that each
$L_i$ is an eigenspace of the map $[[e,f],\cdot]$ with the eigenvalue $i$.
Then $e$, $f$ and $[e,f]$ form an $\Sl_2$-triple, and the maps $[e,\cdot]$
and $[f,\cdot]$ are mutualy inverse isomporphisms of $L_{-1}$ and $L_1$.
Moreover, maps $[x,\cdot]$ with $x$ from $\langle e,f,[e,f]\rangle=\Sl_2$
defines a structure of left $\MM_2(F)$-module on $L_1\oplus L_{-1}$, that
by inspection coincides with the usual structure on $F^2\otimes L_1$
(after identification of $L_{-1}$ and $L_1$ mentioned above).

Now $L$ defines two kind of structures: one is a Lie triple structure on
$L_1\oplus L_{-1}$, and the other is a ternary system considered by
Faulkner in \cite{Fa71} on $L_1$ (roughly speaking, it is an asymmetric
version of a Freudenthal triple system). Namely, define maps
$\langle \cdot,\cdot\rangle$ and $\langle\cdot,\cdot,\cdot\rangle$ by
formulas
\begin{align*}
&[u,v]=\langle u,v\rangle e\\
&\langle u,v,w\rangle=[[[f,u],v],w].
\end{align*}
Note that $\langle\cdot,\cdot\rangle$ allows to identify the dual $L_1^*$
with $L_1$, and so the map
$$
L_1\otimes L_1\to\End(L_1)
$$
corresponding to $\langle\cdot,\cdot,\cdot\rangle$ produces a linear map
$$
\pi\colon\End(L_1)\to\End(L_1),
$$
namely
$$
\pi(\langle\cdot,u\rangle v)=\langle u,v,w\rangle.
$$
By the Morita equivalence, we can consider $\pi$ as a map
$$
\End_{\MM_2(F)}(F^2\otimes L_1)\to\End_{\MM_2(F)}(F^2\otimes L_1).
$$
Also, the same equivalence gives rise to a Hermitian (with respect to the
canonical symplectic involution on $\MM_2(F)$) form
$$
\phi\begin{pmatrix}\begin{pmatrix}u_1\\u_2\end{pmatrix},\begin{pmatrix}v_1\\v_2\end{pmatrix}\end{pmatrix}
=\begin{pmatrix}\langle u_1,v_2\rangle&-\langle u_1,v_1\rangle\\\langle u_2,v_2\rangle&
-\langle u_2,v_1\rangle\end{pmatrix}.
$$

Now we want to relate the two structures on $V_1\oplus V_{-1}\simeq F^2\otimes V_1$.
Direct calculation shows that
\begin{equation}\label{*}
D(u,v)=\frac{1}{2}\big(\pi(\phi(\cdot,u)v-\phi(\cdot,v)u)+\phi(v,u)-\phi(u,v)\big).
\end{equation}
This description admits a Galois descent. Namely, let $Q$ be a quaternion
algebra over $F$, $W$ be a left $Q$-module equipped with a Hermitian (with respect
to the canonical involution on $Q$) form $\phi$ and a linear map
$$
\pi\colon\End_Q(W)\to\End_Q(W).
$$
Assume that $\phi$ and $\pi$ become maps as above over a splitting field of $Q$.
In terms of \cite{Ga01b} this means that $\End_Q(W)$ together with $\pi$ and
the symplectic involution adjoint to $\phi$ form a \emph{gift} (an abbreviation for
a \emph{generalized Freudenthal (or Faulkner) triple}); one can state the
conditions on $\pi$ and $\phi$ as a list of axioms not appealing to the descent
(Garibaldi assumes that $W$ is of dimension $28$ over $Q$, but this really doesn't
matter, at least under some additional restrictions on the characteristic of $F$).
Then equation~\eqref{*} defines on $W$ a structure of a Lie triple system, hence
the embedded Lie algebra $\Der(W)\oplus W$.

\section{$D_6+A_1$-construction}

We say that a map of functors $A\to B$ from fields to sets is \emph{surjective
at $2$} if for any field $F$ and $b\in B(F)$ there exists an odd degree separable
extension $E/F$ and $a\in A(E)$ such that the images of $a$ and $b$ in $B(E)$
coincide.

We enumerate simple roots as in \cite{Bou}. Erasing vertex $1$ from the extended
Dynkin diagram of $E_7$ we see that the simply connected split group $E_7^{sc}$
contains a subgroup of type $D_6+A_1$, namely $(\Spin_{12}\times\SL_2)/\mu_2$. 
Its image in the adjoint group $E_7^{ad}$ is $(\HSpin_{12}\times\SL_2)/\mu_2$,
which we denote by $H$ for brevity.

\begin{thm}\label{thmSurj}
The map $\HH^1(F,H)\to\HH^1(F,E_7^{ad})$ is surjective at $2$.
\end{thm}
\begin{proof}
Note that $W(D_6+A_1)$ and $W(E_7)$ has the same Sylow $2$-subgroup.
Then the result follows by repeating the argument from the proof of
Proposition~14.7, Step 1 in \cite{Ga09} (this argument is a kind of folklore).
\end{proof}

The long exact sequence
$$
\HH^1(F,\mu_2)\to\HH^1(F,H)\to\HH^1(\PGO_{12}^+\times\PGL_2)\to\HH^2(F,\mu_2)
$$
shows that the orbits of $\HH^1(F,H)$ under the action of $\HH^1(F,\mu_2)$
are the isometry classes of central simple algebras of degree $12$ with orthogonal involutions $(A,\sigma)$ and fixed isomorphism
$\Cent(\Cliff_0(A,\sigma))\simeq F\times F$, with $[\Cliff_0^+(A,\sigma)]=[Q]$
in $\Br(F)$ for some quaternion algebra $Q$, where $\Cliff_0$ stands for the
Clifford algebra and $\Cliff_0^+$ for its first component (see \cite[\S~8]{KMRT} for
definitions). Now
$$
\Cliff_0^+(A,\sigma)\simeq\End_Q(W)
$$ for a $16$-dimensional space $W$ over $Q$, and the canonical involution
on $\Cliff_0^+(A,\sigma)$ induces a Hermitian form $\phi$ on $W$ up to a scalar
factor. It is not hard to see that $\HH^1(F,H)$ parametrizes all the mentioned data
together with $\phi$ (and not only its similarity class), and $\HH^1(F,\mu_2)$
multiplies $\phi$ by the respective constant. Over a splitting field of $Q$
the $32$-dimensional half-spin representation carries a structure of Faulkner
ternary system, so we are in the situation of Section~\ref{secLie}. The resulting
embedding $66+3+64=133$-dimensional Lie algebra $\Der(W)\oplus W$ is the
twist of the split Lie algebra of type $E_7$ obtained by a cocycle representing the
image in $\HH^1(F,E_7^{ad})$. Theorem~\ref{thmSurj} shows that any Lie
algebra of type $E_7$ over $F$ arises this way up to an odd degree extension.

Recall that the class of Tits algebra of a cocycle class from $\HH^1(F,E_7^{ad})$
is its image under the connecting map of the long exact sequence
$$
\HH^1(F,E_7^{sc})\to\HH^1(F,E_7^{ad})\to\HH^2(F,\mu_2).
$$
The sequence fits in the following diagram:
$$
\xymatrix{
\HH^1(F,(\Spin_{12}\times\SL_2)/\mu_2)\ar[r]\ar[d]&\HH^1(F,H)\ar[r]\ar[d]&
\HH^2(F,\mu_2)\ar@{=}[d]\\
\HH^1(F,E_7^{sc})\ar[r]&\HH^1(F,E_7^{ad})\ar[r]&\HH^2(F,\mu_2).
}
$$

The fundamental relation for groups of type $D_6$ (see \cite[9.14]{KMRT}) shows
that the class of the Tits algebra in $\Br(F)$ of the class in $\HH^1(F,E_7^{ad})$
corresponding to $\Der(W)\oplus W$ is $[A]+[Q]$.

\section{$7A_1$-construction}

Consider a split simply connected group $E_7^{sc}$ over a field $F$ of characteristic not $2$ and the connected maximal rank subgroup $H'$ of type $7A_1$ there obtained by Borel--de Siebenthal procedure.

\begin{thm}\label{thmSurj2}
The map $\HH^1(F,H')\to\HH^1(F,E_7^{sc})$ is surjective at $2$.
\end{thm}
\begin{proof}
By Theorem~\ref{thmSurj} it suffices to show that the map
$$
\HH^1(F,H')\to\HH^1(F,H)
$$
is surjective at $2$.

Let $\xi$ be an element in $\HH^1(F,H)$; denote by $Q$ the quaternion algebra corresponding
to the image of $\xi$ in $\HH^1(F,\PGL_2)$. Twisting by $Q$, the cocycle class corresponding to $\xi$
comes from an element $\xi'\in\HH^1(F,\Spin(\MM_6(Q),\hyp))$, where $\hyp$ stands for the involution adjoint to the hyperbolic anti-Hermitian form with respect to the canonical involution on $Q$
(see \cite{KMRT} for references).

Denote by $(A,\sigma)$ the algebra with involution $(\MM_6(Q),\hyp)$ twisted by the image of $\xi'$ in $\HH^1(F,\PGO(M_6(Q),\hyp))$. By \cite[Cor.~3.3]{QT14} we have
\begin{equation}\label{eqDecomp}
(A,\sigma)\in(Q_1,\bar{\ })\otimes(H_1,\bar{\ })\boxplus(Q_2,\bar{\ })\otimes(H_2,\bar{\ })
\boxplus(Q_3,\bar{\ })\otimes(H_3,\bar{\ })
\end{equation}
for some quaternion algebras $Q_i$, $H_i$ over $F$ such that $[Q]=[Q_i]+[H_i]$ in $\Br(F)$
and $[H_1]+[H_2]+[H_3]=0$ in $\Br(F)$.
Denote by $H''$ the maximal rank subgroup in $\Spin(\MM_6(Q),\hyp)$ isogeneous to $\Spin(\MM_2(Q),\hyp)^{\times 3}$. Since $H''$ contains the center of $\Spin(\MM_6(Q),\hyp)$, it follows that $\xi'$ comes from $\xi''\in\HH^1(F,H'')$. Twisting back by $Q$ we finally obtain an element
$\xi'''\in\HH^1(F,H')$ whose image is $\xi$.
\end{proof}

We provide now a rational interpretation of the cohomological construction used in the proof. This is a twisted version of Manivel's construction from
\cite[Theorem~4]{Ma06}.

First note that the relations between classes $[Q]$, $[Q_i]$ and $[H_i]$ can be summarized as follows. Put the quaternions at the point of the Fano plane as the picture below shows; we will denote the quaternion algebra corresponding to a vertex $v$ by $Q_v$. Now for any three collinear points $u$, $v$, $w$ we have the relation
$$
[Q_u]+[Q_v]+[Q_w]=0\in\Br(F).
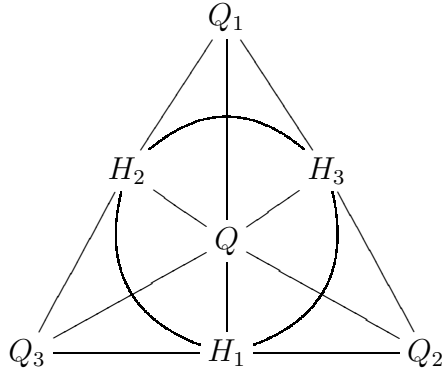
$$

$$
\xymatrix@R=10pt@C=17.32pt{
&&Q_1\ar@{-}[dddl]\ar@{-}[dddr]\ar@{-}[dddd]\\
\\
\\
&H_2\ar@{-}[dddl]\ar@{-}[rd]\ar@/^1.732pc/@{-}[rr]&&H_3\ar@{-}[dddr]\ar@{-}[ld]\ar@/^1.732pc/@{-}[dddl]\\
&&Q\ar@{-}[dd]\ar@{-}[lldd]\ar@{-}[rrdd]\\
\\
Q_3\ar@{-}[rr]&&H_1\ar@/^1.732pc/@{-}[uuul]\ar@{-}[rr]&&Q_2
}
$$

As in \cite{Ma06}, we use a bijection between the lines and the quadruples of points that are not incident to a line. If we have two lines $\alpha$ and $\beta$, then the corresponding quadruples have two points in common, say, $\alpha=(xyuv)$ and $\beta=(xyzw)$, and $\alpha+\beta=(uvzw)$ is the third line such that $\alpha$, $\beta$ and $\alpha+\beta$ are concurrent.

Further, for any line $\alpha=(xyuv)$ we have
$$
[Q_x]+[Q_y]+[Q_u]+[Q_v]=0\in\Br(F),
$$
so
$$
Q_x\otimes Q_y\otimes Q_u\otimes Q_v\simeq\End(V_\alpha)
$$
for a $16$-dimensional $F$-vector space $V_\alpha$.

Note that the group $\tilde H=\prod_v\SL_1(Q_v)$ acts on the left-hand side by algebra homomorphisms, so we obtain a projective representation
$\tilde H\to\PGL(V_\alpha)$ that lifts to a usual representation $\tilde H$ on $V_\alpha$.

Now in the Lie algebra ${}_{\xi}\mathfrak{e}_7$ twisted by the image of $\xi\in\HH^1(F,E_7^{ad})$ there
is a Lie subalgebra $\mathfrak{h}=\vartimes_v\Sl_1(Q_v)$,
and as $\mathfrak{h}$-module $\mathfrak{g}$ decomposes as follows (see \cite[Theorem~4]{Ma06}
in the split case, the general case readily follows by descent):
\begin{equation}\label{eq:decomp}
{}_\xi\mathfrak{e}_7=\mathfrak{h}\oplus\bigoplus_{\alpha}V_\alpha.
\end{equation}

We want to reconstruct the Lie bracket. From the ${\mathbb O}$-grading we see that for any $\alpha=(xyuv)$ the submodule $\mathfrak{h}\oplus V_\alpha$ is actually a Lie subalgebra isomorphic to
$$
\spin((Q_x,\bar{\ })\otimes (Q_y,\bar{\ })\boxplus(Q_u,\bar{\ })\otimes (Q_v,\bar{\ }))
\times\vartimes_{z\ne x,y,u,v}\Sl_1(Q_z)
$$
(see \cite[\S~2]{Ma06} in the split case and use descent in the general case).

Further,
$$
[V_\alpha,V_\beta]\le V_{\alpha+\beta}\text{ for }\alpha\ne\beta.
$$
Explicitly, if $\alpha=(xyuv)$ and $\beta=(xyzw)$, we have a $\tilde H$-equivariant algebra isomorphism
$$
\End(V_\alpha)\otimes\End(V_\beta)\simeq\End(Q_x)\otimes\End(Q_y)\otimes\End(V_{\alpha+\beta}).
$$
This gives rise to a $\tilde H$-equivariant map
$$
\xymatrix{
V_\alpha\otimes V_\beta\ar[r]&Q_x\otimes Q_y\otimes V_{\alpha+\beta}
\ar[rrr]^-{\Trd_{Q_x}\otimes\Trd_{Q_y}\otimes\id}&&&V_{\alpha+\beta}
}
$$
defined up to a constant.

Now we compute the Rost invariant $r(\xi)$ of a cocycle $\xi$ from $\HH^1(F,E_7^{sc})$ that can be
obtained this way (see \cite{Ga09} or \cite[\S~31]{KMRT} for definitions). As in the proof of Theorem~\ref{thmSurj2} we find $Q$, $\xi'\in\HH^1(F,\Spin(\MM_6(Q),\hyp))$ and $(A,\sigma)$.
The embedding $D_6\to E_7$ has the Rost multiplier $1$, hence $r(\xi)=r(\xi')$. In terms of
\cite[\S~2.5]{QT14} we have
\begin{equation}\label{eqRost}
r(\xi)=e_3(A,\sigma)\mod F^\times\cup[Q].
\end{equation}

\section{Tits construction}

Now we reproduce a construction from \cite{Tits90} in our terms. Let $D$ be
an algebra of degree $4$ and $\mu$ be a constant from $F^{\times}$.
By the exceptional isomorphism $A_3=D_3$ the group $\PGL_1(D)$ defines a $3$-dimensional anti-Hermitian form $h$ over $Q$ up to a constant, where $[Q]=2[D]$
in $\Br(F)$. Consider the algebra $\MM_6(Q)$ with the orthogonal involution $\sigma$
adjoint to the $6$-dimensional form $\langle 1,-\mu\rangle\otimes h$. One of the component of
$\Cliff_0(\MM_6(Q),\sigma)$ is trivial in $\Br(F)$ and the other is Brauer-equivalent
to $Q$. Choose some $\phi$ on $W=Q^{16}$. The class of the
Tits algebra of the corresponding cocycle class in $\HH^1(F,E_7^{ad})$ is trivial,
so the cocycle class comes from some $\xi\in\HH^1(F,E_7^{sc})$.

\begin{thm}
For $D$ and $\mu$ as above, there is a cocycle from $\HH^1(F,E_7^{sc})$
whose Rost invariant is $(\mu)\cup[D]$. In particular, if this element cannot be
written as a sum of two symbols from $\HH^3(F,\ZZ/2)$ with a common slot,
then there is a strongly inner anisotropic group of type $E_7$ over $F$.
\end{thm}
\begin{proof}
It follows from \cite[Corollary~2.17]{QT14} that
$$
r(\xi)=(\mu)\cup[D]+(\lambda)\cup[Q]
$$ for some $\lambda$. Changing $\phi$ to
$\lambda\phi$ adds $(\lambda)\cup[Q]$ (cf. \cite[p.~441]{KMRT}), so there is a
cocycle class from $\HH^1(F,E_7^{sc})$ whose Rost invariant is $(\mu)\cup[D]$.
The last claim follows from the easy computation of the Rost invariant of cocycles corresponding to isotropic groups of type $E_7$ (cf. \cite[Appendix~A, Proposition]{Ga09}).
\end{proof}

\noindent
{\sc
Victor Petrov\\
PDMI RAS, Nab. Fontanki, 27, Saint Peterburg, 191023 Russia\\
Chebyshev Laboratory, St. Petersburg State University, 14th Line, 29b, Saint Petersburg, 199178 Russia
}

\medskip

\noindent
{\tt victorapetrov@googlemail.com}


\begin{thebibliography}{9}

\bibitem{AF84} B.N.~Allison, J.R.~Faulkner, {\it A Cayley-Dickson Process for
a Class of Structurable Algebras}, Trans. Amer. Math. Soc. {\bf 283} (1984),
185--210.


\bibitem{Bou} N.~Bourbaki, {\it Groupes et alg\`ebres de Lie. Chapitres 4--6}, Hermann, Paris, 1968.


\bibitem{De98} I.~Dejaiffe, {\it Somme orthogonal d'algebres a involution et algebra de Clifford},
Comm. in Algebra \textbf{26} (1998), 1589--1612.


\bibitem{Br68} R.B.~Brown, {\it A minimal representation for the Lie algebra $E_7$},
Ill. J. Math. {\bf 12} (1968), 190--200.






\bibitem{Fa71} J.R.~Faulkner, {\it A construction of Lie algebras from a class of
ternary algebras}, Trans. Amer. Math. Soc. {\bf 155} (1971), 397--408.


\bibitem{Fr53} H.~Freudenthal, {\it Sur le groupe exceptionnel $E_7$}, Proc.
Nederl. Akad. Wetensch. Ser. A {\bf 56} (1953), 81--89.


\bibitem{Ga01a} S.~Garibaldi, {\it Structurable algebras and groups of type
$E_6$ and $E_7$}, J. Algebra {\bf 236} (2001), 651--691.


\bibitem{Ga01b} S.~Garibaldi, {\it Groups of type $E_7$ over arbitrary fields},
Comm. in Algebra, {\bf 29} (2001), 2689--2710.


\bibitem{Ga09} S.~Garibaldi, {\it Cohomological invariants: Exceptional groups
and Spin groups}, Memoirs of the AMS {\bf 937}, 2009.






\bibitem{KMRT} M.-A.~Knus, A.~Merkurjev, M.~Rost, J.-P.~Tignol,
{\it The book of involutions}, AMS Colloquium Publ. {\bf 44}, 1998.


\bibitem{Ma06} L.~Manivel, {\it Configuration of lines and models of Lie algebras},
J. Algebra {\bf 304} (2006), 457--486.


\bibitem{QT14} A.~Qu\'eguiner--Mathieu, J.-P.~Tignol, {\it The Arason invariant of orthogonal
involutions of degree $12$ and $8$, and quaternionic subgroups of the Brauer group},
preprint, available at
{\tt http://arxiv.org/abs/1406.7705}


\bibitem{Spr} T.~Springer, {\it Linear algebraic groups}, 2nd ed., Birkh\"auser, Boston, 1998.




\bibitem{Tits66a} J.~Tits, {\it Alg\`ebres alternatives, alg\`ebres de Jordan et
alg\`ebres de Lie exceptionelles. I. Construction}, Indag. Math. {\bf 28} (1966), 223--237.


\bibitem{Tits66b} J.~Tits, {\it Classification of algebraic semisimple groups},
Algebraic groups and discontinuous subgroups, Proc. Sympos. Pure Math. {\bf 9},
Amer. Math. Soc., Providence RI, 1966, 33--62.




\bibitem{Tits90} J.~Tits, {\it Strongly inner anisotropic forms of simple
algebraic groups}, J. of Algebra {\bf 131} (1990), 648--677.




\end{thebibliography}
\end{document}